\documentclass[12pt]{article}
\usepackage{amsthm, amsmath, amssymb, latexsym, amscd}
\newcommand{\e}{\epsilon}
\newcommand{\x}{\mathbf{x}}

\newcommand{\g}{\frak{g}}

\newcommand{\Hom}{\mathrm{Hom}}

\newcommand{\Coder}{\mathrm{Coder}}
\newcommand{\Der}{\mathrm{Der}}

\newcommand{\p}{\prime}
\renewcommand{\c}{\circ}
\newcommand{\ot}{\otimes}
\newcommand{\pa}{\partial}

\newtheorem{definition}{Definition}[section]
\newtheorem{claim}[definition]{Claim}
\newtheorem{lemma}[definition]{Lemma}
\newtheorem{proposition}[definition]{Proposition}
\newtheorem{theorem}[definition]{Theorem}
\newtheorem{corollary}[definition]{Corollary}
\newtheorem{remark}[definition]{Remark}
\newtheorem{example}[definition]{Example}
\date{}

\begin{document}

\title{Derived brackets and sh Leibniz algebras}

\author{K. UCHINO}
\maketitle
\abstract
{
We develop a general framework for the construction
of various derived brackets.
We show that suitably deforming the differential
of a graded Leibniz algebra extends the derived
bracket construction and leads to the notion
of strong homotopy (sh) Leibniz algebra.
We discuss the connections among
homotopy algebra theory, deformation theory
and derived brackets.
We prove that the derived bracket construction
induces a map from
suitably defined deformation theory equivalence
classes to the isomorphism classes of sh Leibniz algebras.
}

%%%%%%%%%%%%%%%%%%%%%%%
\section{Introduction.}
%%%%%%%%%%%%%%%%%%%%%%%

Let $(V,d,\{,\})$ be a chain complex
equipped with a binary bilinear $V$-valued operation $\{,\}$.
The triple $(V,d,\{,\})$ is called a dg Leibniz algebra
or a dg Loday algebra by some authors,
if the differential is a derivation with respect to
the bracket and the bracket satisfies
the (graded) Leibniz identity.
When the bracket is anti-commutative,
the Leibniz identity is equivalent to the Jacobi identity.
In this sense, (dg) Leibniz algebras are
noncommutative analogues of classical
(dg) Lie algebras.\\
\indent
Let $(V,d,\{,\})$ be a dg Leibniz algebra.
We define a modified bracket:
$$
\{x,y\}_{d}:=(-1)^{x}\{dx,y\},
$$
which is called a {\em derived bracket}.
In Kosmann-Schwarzbach \cite{K}, it was shown that
the derived bracket satisfies the Leibniz identity.
The original idea of the derived bracket
goes back at least to Koszul  (unpublished).
The derived brackets play important roles
in modern analytical mechanics (cf. \cite{K2}).
For instance, a Poisson bracket on a smooth manifold
is given as a derived bracket,
$\{f,g\}:=[df,g]_{SN}$, where $f,g$ are smooth functions
on the manifold,
$[,]_{SN}$ is a Schouten-Nijenhuis bracket and $d$ is
a coboundary operator of Poisson cohomology.
It is known that the Schouten-Nijenhuis bracket is also
a derived bracket of a certain graded Poisson bracket.
\medskip\\
%This gives a hierarchy of derived brackets,
%which is closely related with a hierarchy
%of various Hamiltonian formalisms
%(classical Hamiltonian-, BV-, AKSZ-formalism and so on).\\
\indent
We consider $n$-fold derived brackets:
$$
(\pm)[[...[\delta x_{1},x_{2}]...],x_{n}],
$$
where $[\cdot,\cdot]$ is a Lie bracket,
$\pm$  an appropriate sign,
and $\delta$ 
a certain derivation, not necessarily of square zero.
The $n$-ary (higher) derived brackets in the category
of {\em Lie} algebras
were studied by several authors in various contexts:
in an article on Poisson geometry by Roytenberg (2002) \cite{R1},
in a paper on homotopy algebra theory by Voronov (2005)
\cite{Vo}, in early work of Vallejo (2001) \cite{Va}
who gave a necessary and sufficient condition
for the $n$-ary derived brackets become Nambu-Lie brackets.\\
\indent
The purpose of this note is to complete
the theory of higher derived bracket construction
in the category of 
%\uuu Leibniz} 
{\em Leibniz\/}
algebras.
To study the higher derived bracket
composed of {\em pure} Leibniz brackets,
we apply the theory of 
{\em sh Leibniz algebras}
(also called Leibniz ${\infty}$-algebras,
sh Loday algebras or Loday ${\infty}$-algebras).
Sh Leibniz algebras are Leibniz algebras
{\em up to homotopy} as well as noncommutative analogues
of sh Lie algebras.
We refer the reader to Ammar and Poncin \cite{AP}
for the study of sh Leibniz algebras.
We give a short survey of
sh Leibniz algebras in Section 4.1 below.
The main result of this note is
Theorem \ref{them}:
Let $(V,\delta_{0},\{,\})$ be a dg Leibniz algebra.
We consider a deformation of $\delta_{0}$,
$$
\delta_{t}=\delta_{0}+t\delta_{1}+t^{2}\delta_{2}
+\cdot\cdot\cdot,
$$
where $t$ is a formal parameter
and $\delta_{t}$ a differential on $V[[t]]$.
We define an $i$-ary derived bracket as
$$
l_{i}(x_{1},...,x_{i})
:=(\pm)\{\{...\{\delta_{i-1}x_{1},x_{2}\},...\},x_{i}\},
$$
where $\pm$ is an appropriate sign.
We prove that the collection of the higher
derived brackets, $\{l_{1},l_{2},...\}$,
yields an sh Leibniz algebra structure.
The theorem follows from a universal formula,
satisfied by Leibniz brackets,
which we establish in Lemma \ref{key1}.\\
\indent
The higher derived bracket construction proposed in this paper
is useful to study a relation between
homotopy algebra theory and deformation theory.
In Proposition \ref{lastprop}, we will show that
if two deformations of $\delta_{0}$ are gauge equivalent,
then the induced sh Leibniz algebras are equivalent;
in other words, the higher derived bracket construction
is invariant under gauge transformations.
%\medskip\\
%\noindent
%Remark.
%In Loday and collaborators works \cite{CLP,L1,L2,L3},
%they study right Leibniz algebras.
%In the following, we study the left Leibniz algebras.
\medskip\\
\noindent
Acknowledgement.
The author would like to thank very
much Professor Jean-Louis Loday
and referees for kind advice and useful comments,
and also thank Professors Johannes Huebschmann
and Akira Yoshioka
for kind advice.
%%%%%%%%%%%%%%%%%%%%%%%
\section{Preliminaries}
%%%%%%%%%%%%%%%%%%%%%%
\subsection{Notation and Assumptions}
The base field is a field $\mathbb{K}$
of characteristic zero. The unadorned tensor product
denotes the tensor product  $\ot:=\ot_{\mathbb{K}}$
over the field $\mathbb{K}$.
We follow the standard Koszul sign convention, for
instance, a linear map $f\ot g:V\ot V\to V\ot V$
satisfies
$$
(f\ot g)(x\ot y)=(-1)^{|g||x|}f(x)\ot g(y),
$$
where $x,y\in V$ and
where $|g|$,$|x|$ are the degrees of $g$,$x$.
We will denote by $s$ the operator 
that raises degree by $1$ and, likewise,
by $s^{-1}$ the operator that lowers degree by $1$.
The Koszul sign convention for shifting operators is,
for instance,
$$
s\ot s=(s\ot 1)(1\ot s)
=-(1\ot s)(s\ot 1).
$$
We call a derivation of degree $1$ a {\em differential},
if it is of square zero.
Given a homogeneous member $x$ of a graded
vector space, we denote the sign $(-1)^{|x|}$ simply by $(-1)^{x}$.

%%%%%%%%%%%%%%%%%%%%%%%%%%%%%%%%%%%%%%%%%%%%%%%%%%
\subsection{Leibniz algebras and derived brackets}
%%%%%%%%%%%%%%%%%%%%%%%%%%%%%%%%%%%%%%%%%%%%%%%%%%
Let $(V,d,\{,\})$ be a %differential graded (dg) vector space 
chain complex
equipped with a binary bracket.
We assume that the degree of the differential is $+1$ (or odd)
and the degree of the bracket is $0$ (or even).
The triple is called a dg (left) Leibniz algebra,
or a dg (left) Loday algebra by some authors,
if $d$ is a derivation with respect to the bracket
and the bracket satisfies a Leibniz identity, i.e.,
\begin{eqnarray*}
d\{x,y\}&=&\{dx,y\}+(-1)^{|x|}\{x,dy\},\\
\{x,\{y,z\}\}&=&\{\{x,y\},z\}+(-1)^{|x||y|}\{y,\{x,z\}\},
\end{eqnarray*}
where $x,y,z\in V$.
A dg Lie algebra can be seen as a special Leibniz algebra
of which the bracket is anti-commutative.
In this sense, (dg) Leibniz algebras are noncommutative
analogues of (dg) Lie algebras.\\
\indent
%\\ \indent
We recall the classical derived bracket construction in \cite{K,K2}.
Define a new bracket on the shifted space $sV$ by
\begin{equation}\label{lasdder}
\{sx,sy\}_{d}:=(-1)^{x}s\{dx,y\}.
\end{equation}
This bracket is called a (binary) derived bracket on $sV$.
Eq. (\ref{lasdder}) is equal to
the following tensor identity,
$$
\{\cdot,\cdot\}_{d}(sx\ot sy)=
s\{\cdot,\cdot\}(s^{-1}\ot s^{-1})(sds^{-1}\ot 1)(sx\ot sy).
$$
We recall two basic propositions.
\begin{itemize}
\item
The derived bracket also satisfies
the graded Leibniz identity, i.e.,
$$
\{sx,\{sy,sz\}_{d}\}_{d}
=\{\{sx,sy\}_{d},sz\}_{d}+(-1)^{(x+1)(y+1)}\{sy,\{sx,sz\}_{d}\}_{d}.
$$
\end{itemize}
We consider the cases of dg Lie algebras.
\begin{itemize}
\item
Let $(V,d,[,])$ be a dg Lie algebra and 
let $\g(\subset V)$ a trivial subalgebra of the Lie algebra.
If $s\g$ is closed under the derived bracket,
then $s\g$ is a Lie algebra, that is,
the derived bracket is anti-commutative on $s\g$.

\end{itemize}
%%%%%%%%%%%%%%%%%%%%%%%
\section{Main results}
%%%%%%%%%%%%%%%%%%%%%%%
Let $V$ be a graded vector space
and let $l_{i}:V^{\ot i}\to V$ be
an $i$-ary multilinear map with the degree $2-i$,
for each $i\ge 1$.
\begin{definition}\label{shdef}
(\cite{AP})
The space $(V,l_{1},l_{2},...)$
with the multilinear maps
is called a strong homotopy (sh)
Leibniz algebra,
if the collection $\{l_{i}\}_{i\ge 1}$
satisfies (\ref{homleib}) below.
\begin{multline}\label{homleib}
\sum_{i+j=Const}
\sum^{i+j-1}_{k=j}\sum_{\sigma}
\chi(\sigma)(-1)^{(k+1-j)(j-1)}(-1)^
{j(x_{\sigma(1)}+\cdot\cdot\cdot+x_{\sigma(k-j)})}\\
l_{i}(x_{\sigma(1)},...,x_{\sigma(k-j)},
l_{j}(x_{\sigma(k+1-j)},...,x_{\sigma(k-1)},x_{k}),
x_{k+1},...,x_{i+j-1})=0,
\end{multline}
where $x_{\cdot}\in V$,
$\sigma\in S_{k-1}$ is a $(k-j,j-1)$-unshuffle (\cite{M1}),
i.e.,
$$
\sigma(1)<\cdot\cdot\cdot<\sigma(k-j), \ \
\sigma(k+1-j)<\cdot\cdot\cdot<\sigma(k-1),
$$
and $\chi(\sigma)$ is an anti-Koszul sign,
$\chi(\sigma):=sgn(\sigma)\e(\sigma)$.
\end{definition}
An sh Lie algebra can be seen as a
special sh Leibniz algebra whose
structures $l_{i\ge 2}$ are skewsymmetric.
\medskip\\
\indent
Let $(V,\{,\})$ be a Leibniz algebra.
We define an $i$-ary bracket associated with the Leibniz
bracket as
$$
N_{i}(x_{1},...,x_{i}):=\{...\{\{x_{1},x_{2}\},x_{3}\},...,x_{i}\}.
$$
It is well-known that $N_{i}$ satisfies
an $i$-ary Leibniz identity, so-called
Nambu-Leibniz identity (cf. \cite{CLP}).
Hence we denote the higher bracket by $N_{\cdot}$.
Let $\Der(V)$ be the space of derivations on the Leibniz algebra.
For any $D\in\Der(V)$, we define a multilinear map as
$$
N_{i}D:=N_{i}(D\ot\overbrace{1\ot\cdot\cdot\cdot\ot 1}^{i-1}),
$$
or equivalently,
$N_{i}D(x_{1},...,x_{i})=
\{...\{\{D(x_{1}),x_{2}\},x_{3}\},...,x_{i}\}$,
in particular, $N_{1}D:=D$.\\
\indent
Let $\delta_{0}\in\Der(V)$ be a differential on the Leibniz algebra.
We consider a formal deformation of $\delta_{0}$,
$$
\delta_{t}:=\delta_{0}+t\delta_{1}+t^{2}\delta_{2}
+\cdot\cdot\cdot.
$$
The deformation $\delta_{t}$ is
a differential on $V[[t]]$, which
is a Leibniz algebra of formal series
with coefficients in $V$.
The differential condition $\delta^{2}_{t}=0$ is equivalent to
the following condition,
\begin{equation}\label{keycond}
\sum_{i+j=Const}\delta_{i}\delta_{j}=0.
\end{equation}
\begin{definition}\label{defin0}
We define an $i$-ary derived bracket on $sV$ as
\begin{equation*}
l_{i}:=(-1)^{\frac{(i-1)(i-2)}{2}}
s\c{N_{i}}\c\mathbf{s}^{-1}(i)\c
(s\delta_{i-1}s^{-1}\ot\mathbf{1}),
\end{equation*}
where
$\textbf{s}^{-1}(i)=\overbrace{s^{-1}\ot\cdot\cdot\cdot\ot s^{-1}}^{i}$,
$\mathbf{1}=\overbrace{1\ot\cdot\cdot\cdot\ot 1}^{i-1}$.
\end{definition}
It is obvious that the degree of the $i$-ary derived bracket
is $2-i$ for each $i\ge 1$.
We see an explicit expression of the higher derived bracket.
\begin{proposition}
For each $i\ge 1$,
the higher derived bracket has the following form on $V$,
\begin{equation*}
(\pm)\{...\{\{\delta_{i-1}x_{1},x_{2}\},x_{3}\},...,x_{i}\}=
s^{-1}l_{i}(sx_{1},...,sx_{i}),
\end{equation*}
where
\begin{eqnarray*}
\pm=\left\{
\begin{array}{ll}
(-1)^{x_{1}+x_{3}+\cdot\cdot\cdot+x_{2n+1}+\cdot\cdot\cdot} & i=\text{even},\\
(-1)^{x_{2}+x_{4}+\cdot\cdot\cdot+x_{2n}+\cdot\cdot\cdot} & i=\text{odd}.
\end{array}
\right.
\end{eqnarray*}
\end{proposition}
\begin{proof}
$l_{i}(sx_{1},...,sx_{i})=
(-1)^{\frac{(i-1)(i-2)}{2}}
s\c{N_{i}}\c\textbf{s}^{-1}(i)\c(s\delta_{i-1}s^{-1}\ot\textbf{1})
(sx_{1}\ot\cdot\cdot\cdot\ot sx_{i})=$
\begin{eqnarray*}
&=&
(\pm)(-1)^{\frac{(i-1)(i-2)}{2}}s\c{N_{i}}\c
\textbf{s}^{-1}(i)\c(s\delta_{i-1}s^{-1}\ot\textbf{1})\c
\textbf{s}(i)
(x_{1}\ot\cdot\cdot\cdot\ot x_{i})\\
&=&
(\pm)(-1)^{\frac{(i-1)(i-2)}{2}}s\c{N_{i}}\c
\textbf{s}^{-1}(i)\c(s\delta_{i-1}\ot\textbf{s}(i-1))
(x_{1}\ot\cdot\cdot\cdot\ot x_{i})\\
&=&
(\pm)(-1)^{\frac{(i-1)(i-2)}{2}}(-1)^{(i-1)}s\c{N_{i}}\c
\textbf{s}^{-1}(i)\c\textbf{s}(i)
(\delta_{i-1}x_{1}\ot\cdot\cdot\cdot\ot x_{i})\\
&=&
(\pm)(-1)^{\frac{(i-1)(i-2)}{2}}(-1)^{(i-1)}
(-1)^{\frac{i(i-1)}{2}}s\c{N_{i}}
(\delta x_{1}\ot\cdot\cdot\cdot\ot x_{i})\\
&=&(\pm)s\{...\{\{\delta_{i-1}x_{1},x_{2}\},x_{3}\},...,x_{i}\},
\end{eqnarray*}
where $\textbf{s}(i):=s\ot\cdots\ot s$ ($i$-times).
\end{proof}
The main result of this note is as follows.
\begin{theorem}\label{them}
The system $(sV,l_{1},l_{2},l_{3},...)$
associated with the higher derived brackets
defined in Definition \ref{defin0}
forms an sh Leibniz algebra.
\end{theorem}
We will give a proof of the theorem in the next section.
We consider the cases of dg Lie algebras.
\begin{corollary}\label{refintro}
Assume that in Theorem \ref{them} $V$ is a Lie algebra.
Let $\g$ be an abelian subalgebra of the Lie algebra.
If $s\g$ is a subalgebra of the induced sh Leibniz algebra,
then $s\g$ becomes an sh Lie algebra.
\end{corollary}
\begin{example}
(Deformation theory, cf. \cite{M4})
Let $(V,\delta_{0},[,])$ be a dg Lie algebra with
a Maurer-Cartan (MC) element
$\theta_{t}:=t\theta_{1}+t^{2}\theta_{2}+\cdot\cdot\cdot$,
which is a solution of the MC-equation:
$$
\delta_{0}\theta_{t}+\frac{1}{2}[\theta_{t},\theta_{t}]=0.
$$
We put $\delta_{i}(-):=[\theta_{i},-]$ for each $i\ge 1$.
Then the collection $\{\delta_{i}\}$ satisfies eq.
(\ref{keycond}) because $\theta_{t}$ is a solution of
the MC-equation.
Therefore an algebraic deformation theory admits
an sh Leibniz algebra structure,
via the higher derived bracket construction.
\end{example}
%%%%%%%%%%%%%%%%%%%%%%%%%%%%%%%%%%%%%
\section{Proof of Theorem \ref{them}}
%%%%%%%%%%%%%%%%%%%%%%%%%%%%%%%%%%%%%
The theorem is given as a corollary of the key lemma
(Lemma \ref{key1} below). To state the lemma,
we recall an alternative definition
of sh Leibniz algebra.
%%%%%%%%%%%%%%%%%%%%%%%%%%%%%%%%%%%%%%%%%%%%%%%%%%%%%%%
\subsection{Sh Leibniz algebras (cf. \cite{AP})}
%%%%%%%%%%%%%%%%%%%%%%%%%%%%%%%%%%%%%%%%%%%%%%%%%%%%%%%
We recall the notion of dual-Leibniz coalgebra (\cite{L2,L3}).
A dual-Leibniz coalgebra is, by definition,
a (graded) vector space equipped with a comultiplication,
$\Delta$, satisfying the identity below.
$$
(1\ot \Delta)\Delta=(\Delta\ot 1)\Delta+
((12)\ot 1)(\Delta\ot 1)\Delta,
$$
where $(12)\in S_{2}$.
We consider the tensor space over a graded vector space:
$$
\bar{T}V:=V\oplus V^{\ot 2}\oplus V^{\ot 3}\oplus\cdot\cdot\cdot.
$$
Define a comultiplication,
$\Delta:\bar{T}V\to\bar{T}V\ot\bar{T}V$,
by $\Delta(V):=0$ and
\begin{equation*}
\Delta(x_{1},...,x_{n+1}):=
\sum^{n}_{i=1}\sum_{\sigma}\e(\sigma)
(x_{\sigma(1)},x_{\sigma(2)},...,x_{\sigma(i)})\ot
(x_{\sigma(i+1)},...,x_{\sigma(n)},x_{n+1}),
\end{equation*}
where $\e(\sigma)$ is a Koszul sign,
$\sigma$ is an $(i,n-i)$-unshuffle
and $(x_{1},...,x_{n+1})\in V^{\ot(n+1)}$.
Then the pair $(\bar{T}V,\Delta)$ becomes the cofree nilpotent
dual-Leibniz coalgebra over $V$.\\
\indent
Let $\Coder(\bar{T}V)$ be the space of coderivations
on the coalgebra, i.e.,
$D^{c}\in\Coder(\bar{T}V)$ satisfies
$$
\Delta D^{c}=(D^{c}\ot 1)\Delta+(1\ot D^{c})\Delta.
$$
By a standard argument, we have
$\Coder(\bar{T}V)\cong\Hom(\bar{T}V,V)$ (cf. \cite{MSS}).
We recall an explicit formula of the isomorphism.
Let $f:V^{\ot i}\to V$ be an $i$-ary linear map.
It is one of the generators in $\Hom(\bar{T}V,V)$.
The coderivation associated with $f$
is defined by $f^{c}(V^{\ot n<i}):=0$ and
\begin{multline*}
f^{c}(x_{1},...,x_{n\ge i}):=
\sum^{n}_{k=i}\sum_{\sigma}
\e(\sigma)(-1)^{|f|(x_{\sigma(1)}+\cdot\cdot\cdot+x_{\sigma(k-i)})}\\
(x_{\sigma(1)},...,x_{\sigma(k-i)},
f(x_{\sigma(k+1-i)},...,x_{\sigma(k-1)},x_{k}),x_{k+1},...,x_{n}),
\end{multline*}
where $\sigma$ is a $(k-i,i-1)$-unshuffle.
The inverse of the mapping $f\mapsto f^{c}$
is the (co)restriction.\\
\indent
If $f,g\in\Hom(\bar{T}V,V)$ are $i$-ary, $j$-ary multilinear
maps respectively, then
$$
[f^{c},g^{c}]=(f,g)^{c},
$$
where $[f^{c},g^{c}]$ is the canonical commutator (Lie bracket)
on $\Coder(\bar{T}V)$
and where $(f,g)$ is an $(i+j-1)$-ary multilinear map.
Since the mapping $f\mapsto f^{c}$ is an isomorphism,
$(f,g)$ defines a Lie bracket on $\Hom(\bar{T}V,V)$.\\
\indent
In the sequel, we will identify
$\Coder(\bar{T}V)$ with $\Hom(\bar{T}V,V)$
as a Lie algebra. We sometimes omit
the superscript ``$c$" from $f^{c}$.\\
\indent
Given a graded vector space $V$,
an $i$-ary  $i$-multilinear $V$-valued operation
$l_{i}$ on $V$ of degree $2-i$
determines 
a degree $1$ 
element in $\Hom(\bar{T}sV,sV)$.
The following proposition provides an alternative
definition of sh Leibniz algebras.
\begin{proposition}(\cite{AP})
Let $V$ be a graded vector space endowed with
a system $\{l_{i}\}_{i\in\mathbb{N}}$ 
of $i$-ary $i$-multilinear $V$-valued operations,
the operation $\{l_{i}\}_{i\in\mathbb{N}}$ having degree $2-i$
and,  for each $i\ge 1$, let %consider the shifted operators
$$
\pa_{i}:=s^{-1}\c l_{i}\c(s\ot\cdot\cdot\cdot\ot s),
$$
by construction of degree $+1$, viewed
as a member of $\Coder(\bar{T}V)$ via the identifications
$\bar{T}V\cong\bar{T}s^{-1}(sV)$.
Define the coderivation $\pa$ by
$$
\pa:=\pa_{1}+\pa_{2}+\cdot\cdot\cdot.
$$
The system $(sV,l_{1},l_{2},...)$ is an sh Leibniz algebra
if and only if
$$
\frac{1}{2}[\pa,\pa]=0
$$
or equivalently, $\pa\pa=0$.
%(More correctly, $[\pa^{c},\pa^{c}]=(\pa,\pa)^{c}=0$.)
\end{proposition}
%%%%%%%%%%%%%%%%%%%%%%
\subsection{The key Lemma}
%%%%%%%%%%%%%%%%%%%%%%
Let $(V,\{,\})$ be a Leibniz algebra.
We consider a collection of maps:
$$
\Der(V)\to\Hom(\bar{T}V,V)\cong\Coder(\bar{T}V),
\ \ D\mapsto N_{i}D\cong N_{i}^{c}D,
$$
where $N_{i}D$ was defined in Section 3
and where $N_{\cdot}^{c}D$
is the coderivation associated with $N_{\cdot}D$.
Theorem \ref{them} is a consequence of the following
\begin{lemma}\label{key1}
For any derivations $D,D^{\p}\in\Der(V)$
and for any $i,j\ge 1$, the following identity holds.
$$
N_{i+j-1}[D,D^{\p}]=(N_{i}D,N_{j}D^{\p}),
$$
or equivalently,
$$
N_{i+j-1}^{c}[D,D^{\p}]=[N_{i}^{c}D,N_{j}^{c}D^{\p}].
$$
\end{lemma}
\begin{proof}
We show the case of $i=1$.
The general case will be shown in Section 6.
We have
\begin{eqnarray*}
N_{j}[D,D^{\p}]&=&N_{j}([D,D^{\p}]\ot 1\ot\cdot\cdot\cdot\ot 1)\\
&=&N_{j}(DD^{\p}\ot 1\ot\cdot\cdot\cdot\ot 1)-
(-1)^{DD^{\p}}N_{j}(D^{\p}D\ot 1\ot\cdot\cdot\cdot\ot 1).
\end{eqnarray*}
By the derivation property, we have
$N_{j}(DD^{\p}\ot 1\ot\cdot\cdot\cdot\ot 1)=$
$$
DN_{j}(D^{\p}\ot 1\ot\cdot\cdot\cdot\ot 1)
-(-1)^{DD^{\p}}\sum^{j}_{k\ge 2}
N_{j}(D^{\p}\ot 1\ot\cdot\cdot\cdot\ot 1\ot D^{(k}
\ot 1\ot \cdot\cdot\cdot\ot 1).
$$
Hence we obtain $N_{j}[D,D^{\p}]=$
$$
DN_{j}(D^{\p}\ot 1\ot\cdot\cdot\cdot\ot 1)
-(-1)^{DD^{\p}}\sum^{j}_{k\ge 1}
N_{j}(D^{\p}\ot 1\ot\cdot\cdot\cdot\ot 1\ot
D^{(k}\ot 1\ot \cdot\cdot\cdot\ot 1),
$$
which is equal to
$N_{j}[D,D^{\p}]=(N_{1}D,N_{j}D^{\p})$ because $N_{1}D=D$.
\end{proof}
The higher derived brackets are elements in $\Hom(\bar{T}sV,sV)$.
Hence they correspond to the coderivations in $\Coder(\bar{T}V)$,
via the maps,
$$
\Hom(\bar{T}sV,sV)\overset{\text{shift}}{\sim}
\Hom(\bar{T}V,V)\cong\Coder(\bar{T}V).
$$
\begin{lemma}\label{key2}
Let $\pa_{i}$ be the coderivation associated with
the $i$-ary derived bracket.
It has the following form,
$$
\pa_{i}=N_{i}^{c}\delta_{i-1}.
$$
\end{lemma}
\begin{proof}
\begin{eqnarray*}
\pa_{i}&:=&s^{-1}\c l_{i}\c(s\ot\cdot\cdot\cdot\ot s)\\
&=&(-1)^{\frac{(i-1)(i-2)}{2}}
N_{i}\c(s^{-1}\ot\cdot\cdot\cdot\ot s^{-1})\c
(s\delta_{i-1}\ot s\ot\cdot\cdot\cdot\ot s)\\
&=&(-1)^{\frac{(i-1)(i-2)}{2}}N_{i}\c
(\delta_{i-1}\ot s^{-1}\ot\cdot\cdot\cdot\ot s^{-1})\c
(1\ot s\ot\cdot\cdot\cdot\ot s)\\
&=&N_{i}\delta_{i-1}.
\end{eqnarray*}
Hence $\pa_{i}=N_{i}^{c}\delta_{i-1}$ as a coderivation.
\end{proof}
Now, we give a proof of Theorem \ref{them}.
\begin{proof}
By Lemma \ref{key2}, the differential
$\delta_{t}=\sum t^{i}\delta_{i}$
corresponds to the coderivation:
$$
\pa:=\pa_{1}+\pa_{2}+\pa_{3}+\cdot\cdot\cdot.
$$
By Lemmas \ref{key1},
the deformation condition $[d,d]=0$
corresponds to the homotopy algebra condition,
$$
\sum_{i+j=Const}[\pa_{i},\pa_{j}]=
\sum_{i+j=Const}[N_{i}^{c}\delta_{i-1},N_{j}^{c}\delta_{j-1}]=
N_{i+j-1}^{c}\sum_{i+j=Const}[\delta_{i-1},\delta_{j-1}]=0.
$$
\end{proof}
\begin{remark}(cf. Lemma \ref{key1})
We consider the case of the trivial deformation, that is,
$\delta_{t}=t\delta_{1}$.
In this case, the induced sh Leibniz algebra
is an ordinary Leibniz algebra.
We put $CL^{n}(sV):=\Hom(V^{\ot n},V)$ and $b(-):=(\pa_{2},-)$.
Then $(CL^{*}(sV),b)$ is the Leibniz
cohomology complex (\cite{L1}).
The key Lemma implies that $\Der(V)$ provides
a subcomplex of the Leibniz complex:
$$
N_{i}\Der(V)\subset CL^{i}(sV),
$$
because
$(\pa_{2},N_{i}D)=(N_{2}\delta_{1},N_{i}D)=N_{i+1}[\delta_{1},D]$.
If $\delta_{1}$ is an adjoint representation, i.e.,
$\delta_{1}:=ad(\theta):=[\theta,-]$ for some $\theta\in V$,
then $N_{i}ad(V)$ is also a subcomplex,
$$
N_{i}ad(V)\subset N_{i}\Der(V)\subset CL^{i}(sV).
$$
\end{remark}
%%%%%%%%%%%%%%%%%%%%%%%%%%%%
\section{Deformation theory}
%%%%%%%%%%%%%%%%%%%%%%%%%%%%
In this section, we discuss the connection
between deformation theory and sh Leibniz algebras.
The deformation $\delta_{t}$
is considered to be a differential on $V[[t]]$,
which is a Leibniz algebra of formal series
with coefficients in $V$. Let $t\xi_{1}\in\Der(V[[t]])$
be a derivation with the degree $0$.
We consider a transformation,
$$
\delta_{t}^{\p}:=exp(X_{t\xi_{1}})(\delta_{t}),
$$
where $X_{t\xi_{1}}:=[\cdot,t\xi_{1}]$.
By a standard argument,
$\delta_{t}^{\p}$ is also a deformation of $\delta_{0}$.
We have
\begin{eqnarray*}
\delta^{\p}_{0}&=&\delta_{0},\\
\delta^{\p}_{1}&=&\delta_{1}+[\delta_{0},\xi_{1}],\\
\delta^{\p}_{2}&=&\delta_{2}+[\delta_{1},\xi_{1}]+
\frac{1}{2!}[[\delta_{0},\xi_{1}],\xi_{1}],\\
...&...&...\\
\delta^{\p}_{i}&=&\sum^{i}_{n=0}\frac{1}{(i-n)!}X^{i-n}_{\xi_{1}}(\delta_{n}).
\end{eqnarray*}
The collection $\{\delta^{\p}_{i}\}_{i\in\mathbb{N}}$
induces an sh Leibniz algebra structure
$\pa^{\p}=\sum \pa^{\p}_{i}$,
via the higher derived bracket construction.
From Lemmas \ref{key1}, \ref{key2}, we have
$$
\pa^{\p}_{i+1}=N_{i+1}^{c}\delta^{\p}_{i}=
\sum^{i}_{n=0}\frac{1}{(i-n)!}X^{i-n}_{N_{2}^{c}\xi_{1}}(\pa_{n+1}).
$$
Therefore we obtain
$$
\pa^{\p}=exp(X_{N_{2}^{c}\xi_{1}})(\pa),
$$
which implies that $\pa^{\p}$ is equivalent to $\pa$.
We consider a general case.
Let $\xi_{t}:=t\xi_{1}+t^{2}\xi_{2}+\cdot\cdot\cdot$
be a derivation on $V[[t]]$ with degree $0$.
The transformation (\ref{defgauge}) below
is called a {\em gauge transformation}.
\begin{equation}\label{defgauge}
\delta_{t}^{\p}:=exp(X_{\xi_{t}})(\delta_{t}).
\end{equation}
\begin{proposition}\label{lastprop}
(I) If two deformations of $\delta_{0}$
are gauge equivalent, or related via the gauge transformation,
then the induced sh Leibniz algebra structures are equivalent
to each other, i.e., the codifferential $\pa^{\p}$
induced by $\delta_{t}^{\p}$
is related with $\pa$ via the transformation,
\begin{equation}\label{eqpa}
\pa^{\p}=exp(X_{\Xi})(\pa),
\end{equation}
where $\Xi$ is a coderivation,
$$
\Xi:=N_{2}^{c}\xi_{1}+N_{3}^{c}\xi_{2}+\cdot\cdot\cdot
+N_{i+1}^{c}\xi_{i}+\cdot\cdot\cdot.
$$
(II) The exponential of $\Xi$,
$$
e^{\Xi}:=1+\Xi+\frac{1}{2!}\Xi^{2}
+\cdot\cdot\cdot,
$$
is a dg coalgebra isomorphism between
$(\bar{T}V,\pa)$ and $(\bar{T}V,\pa^{\p})$, namely,
(\ref{rlas1}) and (\ref{rlas2}) below hold. 
\begin{eqnarray}
\label{rlas1}\pa^{\p}&=&e^{-\Xi}\cdot\pa\cdot e^{\Xi},\\
\label{rlas2}\Delta e^{\Xi}&=&(e^{\Xi}\ot e^{\Xi})\Delta.
\end{eqnarray}
\end{proposition}
The notion of sh Leibniz algebra homomorphism
is defined to be a map satisfying (\ref{rlas1}) and (\ref{rlas2}).
Thus (II) says that $e^{\Xi}$ is an sh Leibniz algebra
isomorphism.
\begin{proof}
(I) From (\ref{defgauge}) we have
$$
\delta^{\p}_{n}=\delta_{n}+\sum_{n=i+j}[\delta_{i},\xi_{j}]
+\frac{1}{2!}\sum_{n=i+j+k}[[\delta_{i},\xi_{j}],\xi_{k}]+
\cdot\cdot\cdot.
$$
Hence we obtain $\pa^{\p}_{n+1}=N_{n+1}^{c}\delta^{\p}_{n}=$
\begin{multline*}
N_{n+1}^{c}\delta_{n}+
\sum_{n=i+j}N_{n+1}^{c}[\delta_{i},\xi_{j}]
+\frac{1}{2!}\sum_{n=i+j+k}
N_{n+1}^{c}[[\delta_{i},\xi_{j}],\xi_{k}]+\cdot\cdot\cdot=\\
\pa_{n+1}+\sum_{n=i+j}[\pa_{i+1},N_{j+1}^{c}\xi_{j}]
+\frac{1}{2!}\sum_{n=i+j+k}
[[\pa_{i+1},N_{j+1}^{c}\xi_{j}],N_{k+1}^{c}\xi_{k}]
+\cdot\cdot\cdot.
\end{multline*}
This gives (\ref{eqpa}).
\medskip\\
\noindent
(II)
The exponential $e^{\Xi}$ is well-defined as an isomorphism
on $\bar{T}V$, because
$e^{\Xi}$ is finite on $V^{\ot n}$ for each $n$.
For instance, on $V^{\ot 3}$,
$$
e^{\Xi}\equiv
1+(N_{2}^{c}\xi_{1}+N_{3}^{c}\xi_{2})+\frac{1}{2}(N_{2}^{c}\xi_{1})^{2}.
$$
By a direct computation, one can prove that
$$
exp(X_{\Xi})(\pa)=e^{-\Xi}\cdot\pa\cdot e^{\Xi}.
$$
Thus (\ref{rlas1}) holds.
Since $\Xi$ is a coderivation,
$e^{\Xi}$ satisfies (\ref{rlas2}).
\end{proof}
%%%%%%%%%%%%%%%%%%%%%%%%%%%%%%%%%%%
\section{Proof of Lemma \ref{key1}}
%%%%%%%%%%%%%%%%%%%%%%%%%%%%%%%%%%%
\begin{claim}\label{fclaim}
Let $f:V^{\ot i}\to V$ be an $i$-ary linear map.
For each $n$, we define $f^{(k)}:V^{\ot n}\to V^{\ot(n-i+1)}$ by
\begin{multline*}
f^{(k)}(x_{1},...,x_{n}):=
\sum_{\sigma}
\e(\sigma)(-1)^{|f|(x_{\sigma(1)}+\cdot\cdot\cdot+x_{\sigma(k-i)})}\\
(x_{\sigma(1)},...,x_{\sigma(k-i)},
f(x_{\sigma(k+1-i)},...,x_{\sigma(k-1)},x_{k}),x_{k+1},...,x_{n}).
\end{multline*}
Then the coderivation associated with $f$
decomposes as :
$$
f^{c}=\sum_{k\ge i}f^{(k)}.
$$
\end{claim}
In Section 4.2, we established the lemma for $i=1$.
We assume the identity of the lemma
and prove the case of $i+1$, i.e.,
$N_{i+j}[D,D^{\p}]=(N_{i+1}D,N_{j}D^{\p})$,
or equivalently,
$N^{c}_{i+j}[D,D^{\p}]=[N_{i+1}^{c}D,N_{j}^{c}D^{\p}]$.\\
\indent
We put $\x:=(x_{1},...,x_{i+j-1})$.
From the definition of $N_{\cdot}D$, we have
$$
N_{i+j}^{c}[D,D^{\p}](\x,x_{i+j})=
\{N_{i+j-1}^{c}[D,D^{\p}](\x),x_{i+j}\}.
$$
The assumption of the induction yields that
\begin{eqnarray*}
N_{i+j}^{c}[D,D^{\p}](\x,x_{i+j})&=&
\{[N_{i}^{c}D,N_{j}^{c}D^{\p}](\x),x_{i+j}\}\\
&=&\{N_{i}^{c}D\c N_{j}^{c}D^{\p}(\x),x_{i+j}\}
-(-1)^{DD^{\p}}\{N_{j}^{c}D^{\p}\c N_{i}^{c}D(\x),x_{i+j}\}.
\end{eqnarray*}
Claim \ref{fclaim} derives
$$
N_{j}^{c}D^{\p}=\sum_{k\ge j}N_{j}^{(k)}D^{\p},
$$
which gives
\begin{equation}\label{railgun}
N_{i+j}^{c}[D,D^{\p}](\x,x_{i+j})=
\sum^{i+j-1}_{k=j}\{N_{i}^{c}D\c N_{j}^{(k)}D^{\p}(\x),x_{i+j}\}
-(-1)^{DD^{\p}}\{N_{j}^{c}D^{\p}\c N_{i}^{c}D(\x),x_{i+j}\}.
\end{equation}
The first term of (\ref{railgun}) becomes
\begin{eqnarray*}
\sum^{i+j-1}_{k=j}\{N_{i}^{c}D\c N_{j}^{(k)}D^{\p}(\x),x_{i+j}\}&=&
\sum^{i+j-1}_{k=j}N_{i+1}^{c}D\c N_{j}^{(k)}D^{\p}(\x,x_{i+j})\\
&=&
N_{i+1}^{c}D\c N_{j}^{c}D^{\p}(\x,x_{i+j})
-N_{i+1}^{c}D\c N_{j}^{(i+j)}D^{\p}(\x,x_{i+j}),
\end{eqnarray*}
because the coderivation preserves the position of
the most right component $x_{i+j}$.
So it suffices to show that
\begin{multline}\label{ap1}
-(-1)^{DD^{\p}}\{N_{j}^{c}D^{\p}\c N_{i}^{c}D(\x),x_{i+j}\}
=\\
N_{i+1}^{c}D\c N_{j}^{(i+j)}D^{\p}(\x,x_{i+j})
-(-1)^{DD^{\p}}N_{j}^{c}D^{\p}\c N_{i+1}^{c}D(\x,x_{i+j}).
\end{multline}
We need a lemma.
\begin{lemma}\label{themlem1}
For any elements in the Leibniz algebra,
$A,B,y_{1},...,y_{n}\in V$,
\begin{multline*}
N_{n+2}(A,B,y_{1},...,y_{n})=
-(-1)^{AB}\{B,N_{n+1}(A,y_{1},...,y_{n})\}+\\
\sum^{n}_{a=1}(-1)^{B(y_{1}+\cdot\cdot\cdot+y_{a-1})}
N_{n+1}(A,y_{1},...,y_{a-1},\{B,y_{a}\},
y_{a+1},...,y_{n}).
\end{multline*}
\end{lemma}
\begin{proof}
We show the case of $n=2$.
Up to sign,
\begin{eqnarray*}
\{B,\{\{A,y_{1}\},y_{2}\}\}&=&
\{\{B,\{A,y_{1}\}\},y_{2}\}+\{\{A,y_{1}\},\{B,y_{2}\}\}\\
&=&
\{\{\{B,A\},y_{1}\},y_{2}\}+\{\{A,\{B,y_{1}\}\},y_{2}\}
+\{\{A,y_{1}\},\{B,y_{2}\}\}\\
&=&
-\{\{\{A,B\},y_{1}\},y_{2}\}+\{\{A,\{B,y_{1}\}\},y_{2}\}
+\{\{A,y_{1}\},\{B,y_{2}\}\},
\end{eqnarray*}
where $-\{\{A,B\},y_{1}\}=\{\{B,A\},y_{1}\}$ is used.
Thus we obtain
$$
\{B,N_{3}(A,y_{1},y_{2})\}=-N_{4}(A,B,y_{1},y_{2})+
N_{3}(A,\{B,y_{1}\},y_{2})+N_{3}(A,y_{1},\{B,y_{2}\}).
$$
\end{proof}
We prove (\ref{ap1}).
By the definition of coderivation,
\begin{multline*}
N_{i}^{c}D(\x)=
\sum^{i+j-1}_{k=i}\sum_{\sigma}E(\sigma,k-i)\\
(x_{\sigma(1)},...,x_{\sigma(k-i)},N_{i}(Dx_{\sigma(k+1-i)
},...,x_{\sigma(k-1)},x_{k}),x_{k+1},...,x_{i+j-1}),
\end{multline*}
where
$$
E(\sigma,*):=
\e(\sigma)(-1)^{D(x_{\sigma(1)}+\cdot\cdot\cdot+x_{\sigma(*)})}.
$$
Since $N_{n}(x_{1},...,x_{n})=\{\{\{x_{1},x_{2}\},..,\},x_{n}\}$,
$$
N_{n}(x_{1},...,x_{n})=N_{n-i+1}(N_{i}(x_{1},...,x_{i}),x_{i+1},...,x_{n}),
$$
which gives
$S:=-(-1)^{DD^{\p}}\{N_{j}^{c}D^{\p}\c N_{i}^{c}D(\x),x_{i+j}\}=$
\begin{multline*}
-(-1)^{DD^{\p}}\sum^{i+j-1}_{k=i}\sum_{\sigma}E(\sigma,k-i)\\
N_{i+j-k+2}
\big(N_{k-i}(D^{\p}x_{\sigma(1)},...,x_{\sigma(k-i)}),N_{i}(Dx_{\sigma(k+1-i)
},...,x_{\sigma(k-1)},x_{k}),x_{k+1},...,x_{i+j}\big).
\end{multline*}
We put $A:=N_{k-i}(D^{\p}x_{\sigma(1)},...,x_{\sigma(k-i)})$
and $B:=N_{i}(Dx_{\sigma(k+1-i)
},...,x_{\sigma(k-1)},x_{k})$,
then from Lemma \ref{themlem1},
\begin{equation}\label{s=t+u}
S=T+U,
\end{equation}
where
\begin{multline*}
T:=-(-1)^{DD^{\p}}\sum^{i+j-1}_{k=i}\sum_{\sigma}E(\sigma,k-i)E_{1}\\
N_{i+1}\big(Dx_{\sigma(k+1-i)},...,x_{\sigma(k-1)},x_{k},
N_{j}(D^{\p}x_{\sigma(1)},...,x_{\sigma(k-i)},
x_{k+1},...,x_{i+j})\big),
\end{multline*}
\begin{multline*}
U:=-(-1)^{DD^{\p}}\sum^{i+j-1}_{k=i}\sum_{\sigma}
\sum^{i+j-k}_{a=1}E(\sigma,k-i)E_{2}\\
N_{j}(D^{\p}x_{\sigma(1)},...,x_{\sigma(k-i)},
x_{k+1},...,x_{k+a-1},
N_{i+1}(Dx_{\sigma(k+1-i)},...,
x_{\sigma(k-1)},x_{k},x_{k+a}),x_{k+a+1},...,x_{i+j}),
\end{multline*}
where $E_{1}$ and $E_{2}$ are appropriate signs
given by the manner in the lemma above.
\medskip\\
\noindent
\textbf{(I)}
We show the identity,
\begin{equation}\label{newdeft}
T=N_{i+1}^{c}D\c N_{j}^{(i+j)}D^{\p}(\x,x_{i+j}).
\end{equation}
We replace $\sigma$ in T with
an unshuffle permutation $\tau$ along the table,
\begin{center}
\begin{tabular}{ccccccc}
\hline
$\sigma(k+1-i)$&...&$\sigma(k-1)$&$k$&$\sigma(1)$
&...&$\sigma(k-i)$ \\
\hline
$\tau(1)$&...&$\tau(i-1)$&$\tau(i)$&$\tau(i+1)$
&...&$\tau(k)$ \\
\hline
\end{tabular}
\end{center}
Then Koszul sign is replaced with $\e(\tau)$:
$$
\e(\tau)=\e(\sigma)
(-1)^{(x_{\sigma(1)}+\cdot\cdot\cdot+x_{\sigma(k-i)})
(x_{\sigma(k+1-i)}+\cdot\cdot\cdot+x_{\sigma(k-1)}+x_{k})},
$$
and then $E(\sigma,k-i)E_{1}=$
\begin{eqnarray*}
&=&-\e(\sigma)(-1)^{D(x_{\sigma(1)}+\cdot\cdot\cdot+x_{\sigma(k-i)})}
(-1)^{AB}\\
&=&-\e(\sigma)(-1)^{D(x_{\sigma(1)}+\cdot\cdot\cdot+x_{\sigma(k-i)})}
(-1)^{(x_{\sigma(1)}+\cdot\cdot\cdot+x_{\sigma(k-i)}+D^{\p})
(x_{\sigma(k+1-i)}+\cdot\cdot\cdot+x_{\sigma(k-1)}+x_{k}+D)}\\
&=&-\e(\sigma)
(-1)^{(x_{\sigma(1)}+\cdot\cdot\cdot+x_{\sigma(k-i)})(x_{\sigma(k+1-i)}+\cdot\cdot\cdot+x_{\sigma(k-1)}+x_{k})}
(-1)^{D^{\p}(x_{\sigma(k+1-i)}+\cdot\cdot\cdot+x_{\sigma(k-1)}+x_{k})+DD^{\p}}\\
&=&-\e(\tau)
(-1)^{D^{\p}(x_{\sigma(k+1-i)}+\cdot\cdot\cdot+x_{\sigma(k-1)}+x_{k})+DD^{\p}}\\
&=&-\e(\tau)
(-1)^{D^{\p}(x_{\tau(1)}+\cdot\cdot\cdot+x_{\tau(i-1)}+x_{\tau(i)})+DD^{\p}}=
-E^{\p}(\tau,i)(-1)^{DD^{\p}}.
\end{eqnarray*}
Thus T is equal to
$$
T^{\p}:=\sum^{i+j-1}_{k=i}\sum_{\tau}E^{\p}(\tau,i)
N_{i+1}\big(Dx_{\tau(1)},...,x_{\tau(i-1)},x_{\tau(i)=k},
N_{j}(D^{\p}x_{\tau(i+1)},...,x_{\tau(k)},
x_{k+1},...,x_{i+j})\big),
$$
where $\tau$ is an $(i,k-i)$-unshuffle
such that $\tau(i)=k$.
\begin{claim}
$T^{\p}=T^{\p\p}$, where
$$
T^{\p\p}:=\sum_{\nu}E^{\p}(\nu,i)
N_{i+1}\big(Dx_{\nu(1)},...,x_{\nu(i-1)},x_{\nu(i)},
N_{j}(D^{\p}x_{\nu(i+1)},...,x_{\nu(i+j-1)},x_{i+j})\big),
$$
where $\nu$ is an $(i,j-1)$-unshuffle.
\end{claim}
\begin{proof}
We put $k:=\nu(i)$ in $T^{\p\p}$.
Since $\nu$ is an $(i,j-1)$-unshuffle,
$i\le k\le i+j-1$.
Replace $\nu$ with $\tau$.
This replacement preserves the order of variables.
Hence $E^{\p}(\tau,i)=E^{\p}(\nu,i)$, which gives
the identity of the claim.
\end{proof}
Since
$T^{\p\p}
=N_{i+1}^{c}D\c N_{j}^{(i+j)}D^{\p}(\x,x_{i+j})$,
we obtain (\ref{newdeft}).
\medskip\\
\noindent\textbf{(II)}
We show the identity,
\begin{equation}\label{newdefu}
U=-(-1)^{DD^{\p}}N_{j}^{c}D^{\p}\c N_{i+1}^{c}D(\x,x_{i+j}).
\end{equation}
We replace $\sigma$ in U
with an unshuffle permutation $\tau$ along the table,
\begin{center}
\begin{tabular}{cccccc}
\hline
$\sigma(1)$&...&$\sigma(k-i)$&$k+1$&...&$k+a-1$ \\
\hline
$\tau(1)$&...&$\tau(k-i)$&$\tau(k+1-i)$&...&$\tau(k+a-1-i)$ \\
\hline\hline
$\sigma(k+1-i)$&...&$\sigma(k-1)$&$k$& & \\ \hline
$\tau(k+a-i)$&...&$\tau(k+a-2)$&$\tau(k+a-1)$& & \\
\hline
\end{tabular}
\end{center}
Then the Koszul sign is replaced with $\e(\tau)$:
$$
\e(\tau)=\e(\sigma)
(-1)^{(x_{\sigma(k+1-i)}+\cdot\cdot\cdot+x_{\sigma(k-1)}+
x_{k})(x_{k+1}+\cdot\cdot\cdot+x_{k+a-1})},
$$
and then $E(\sigma,k-i)E_{2}=$
\begin{eqnarray*}
&=&\e(\sigma)(-1)^{D(x_{\sigma(1)}+\cdot\cdot\cdot+x_{\sigma(k-i)})}
(-1)^{B(x_{k+1}+\cdot\cdot\cdot+x_{k+a-1})}\\
&=&\e(\sigma)(-1)^{D(x_{\sigma(1)}+\cdot\cdot\cdot+x_{\sigma(k-i)})}
(-1)^{(x_{\sigma(k+1-i)}+\cdot\cdot\cdot+x_{\sigma(k-1)}
+x_{k}+D)(x_{k+1}+\cdot\cdot\cdot+x_{k+a-1})}\\
&=&\e(\sigma)
(-1)^{(x_{\sigma(k+1-i)}+\cdot\cdot\cdot+x_{\sigma(k-1)}+
x_{k})(x_{k+1}+\cdot\cdot\cdot+x_{k+a-1})}
(-1)^{D(x_{\sigma(1)}+\cdot\cdot\cdot+x_{\sigma(k-i)}
+x_{k+1}+\cdot\cdot\cdot+x_{k+a-1})}\\
&=&\e(\tau)(-1)^{D(x_{\tau(1)}+\cdot\cdot\cdot+x_{\tau(k+a-1-i)})}\\
&=&E(\tau,k+a-1-i)=E(\tau,m-i),
\end{eqnarray*}
where $m:=k+a-1$.
\begin{claim}
$U=U^{\p}$, where
\begin{multline*}
U^{\p}:=-(-1)^{DD^{\p}}\sum^{i+j-1}_{m=i}\sum_{\tau}E(\tau,m-i)\\
N_{j}(D^{\p}x_{\tau(1)},...,x_{\tau(m-i)},
N_{i+1}(Dx_{\tau(m+1-i)},...,x_{\tau(m)},
x_{m+1}),x_{m+2},...,x_{i+j}),
\end{multline*}
where $\tau$ is an $(m-i,i)$-unshuffle.
\end{claim}
\begin{proof}
Let $\tau$ be an $(m-i,i)$-unshuffle.
We put $k:=\tau(m)$ and $a:=m+1-\tau(m)$.
Then we have
\begin{multline*}
(\tau(1),...,\tau(m-i);\tau(m+1-i),...,\tau(m),m+1,...,i+j)=\\
(\tau(1),...,\tau(k-i),k+1,...,k+a-1;
\tau(k+a-i),...,\tau(k+a-2),k,k+a,...,i+j).
\end{multline*}
One can replace $\tau$ with an unshuffle $\sigma$,
\begin{multline*}
(\tau(1),...,\tau(m-i);\tau(m+1-i),...,\tau(m),m+1,...,i+j)=\\
(\sigma(1),...,\sigma(k-i),k+1,...,k+a-1;
\sigma(k+1-i),...,\sigma(k-1),k,k+a,...,i+j),
\end{multline*}
which gives the table above.
Up to this permutation, we obtain
$$
\sum_{\tau}=\sum_{(k,a)}\sum_{\sigma}
$$
where $(m-i,m)$ is fixed
and $(k,a)$ runs over all possible pairs.
This gives
$$
\sum_{m\ge i}\sum_{\tau}=\sum_{k\ge i}\sum_{a\ge 1}\sum_{\sigma},
$$
which implies the identity of the claim.
\end{proof}
Since $U^{\p}=-(-1)^{DD^{\p}}N_{j}^{c}D^{\p}\c N_{i+1}^{c}D(\x,x_{i+j})$,
we obtain (\ref{newdefu}).
From (\ref{s=t+u}), (\ref{newdeft}) and (\ref{newdefu}),
we get the desired identity (\ref{ap1}).
The proof is completed.

\noindent
Post doctoral.\\
Science University of Tokyo.\\
3-14-1 Shinjyuku Tokyo Japan.\\
e-mail: K\underline{ }Uchino[at]oct.rikadai.jp
\end{document}